\documentclass[twoside,12pt]{article}

\usepackage{amsmath}
\usepackage{amssymb}
\usepackage{pxfonts}
\usepackage{graphicx}
\usepackage{eucal}
\usepackage{mathrsfs}
\usepackage{theorem}
\usepackage{pifont}
\usepackage{amscd}

\usepackage{sectsty}
\usepackage{pseudocode}
\usepackage{color}
\usepackage{fancyhdr}
\usepackage{framed}
\usepackage[all]{xy}
\usepackage{hyperref}

\definecolor{shadecolor}{rgb}{0.8,0.8,0.8}

\theoremheaderfont{\fontfamily{pzc}\bfseries\large}
\newtheorem{theorem}{Theorem}[section]
\newtheorem{lemma}{Lemma}[section]
\newtheorem{proposition}{Proposition}[section]

\newtheorem{corollary}{Corollary}[section]
\newtheorem{definition}{Definition}[section]

{\theorembodyfont{\rmfamily}

\newenvironment{proof}{{\flushleft \emph{Proof}:}}{\hfill\ding{110}}

\newcommand{\secref}[1]{Section~\ref{#1}}

\newcommand{\thmref}[1]{Theorem~\ref{#1}}

\newcommand{\propref}[1]{Proposition~\ref{#1}}
\newcommand{\lemref}[1]{Lemma~\ref{#1}}
\newcommand{\corref}[1]{Corollary~\ref{#1}}

\newcommand{\tr}{\operatorname{tr}}
\newcommand{\dist}{\operatorname{dist}}
\newcommand{\strain}[1]{\operatorname{Strain}(#1)}
\newcommand{\strainEuc}[1]{\operatorname{Strain}^{\text{Euc}}(#1)}

\newcommand{\R}{\mathbb{R}}
\newcommand{\Sig}{\Sigma}
\newcommand{\sig}{\sigma}
\newcommand{\GG}{g_I}
\newcommand{\Emph}[1]{{\slshape\bfseries #1}} 
\newcommand{\conn}[2]{{\omega^{#1}}_{#2}}
\newcommand{\G}{g}
\newcommand{\cof}[1]{\vartheta^{#1}}
\newcommand{\inv}{\operatorname{inv}}
\newcommand{\fraka}{\mathfrak{a}}

\newcommand{\til}{\tilde}

\newcommand{\brk}[1]{\left(#1\right)}          % \brk{.}     => (.)
\newcommand{\sAverage}[1]{\langle#1\rangle}      % \Average{.} => <.>

\newcommand{\beq}{\begin{equation}}
\newcommand{\eeq}{\end{equation}}

\newcommand{\textand}{\quad\text{ and }\quad}
\newcommand{\Textand}{\qquad\text{ and }\qquad}

\providecommand{\half}{\frac{1}{2}}

\newcommand{\distSO}[1]{\text{dist}(#1,\SOn)}
\newcommand{\tildistSO}[1]{\text{dist}_{\til d}(#1,\SOn)}

\newcommand{\IPF}[2]{\sAverage{#1,#2}_F}
\newcommand{\al}{\alpha}
\newcommand{\be}{\beta}
\newcommand{\ga}{\gamma}

\newcommand{\diag}{\text{diag}}

\newcommand{\SOn}{\text{SO}_n}
\newcommand{\On}{\text{O}_n}
\newcommand{\GLp}{\text{GL}_n^+}
\newcommand{\GLm}{\text{GL}_n^-}
\newcommand{\GL}{\text{GL}_n}

\newcommand{\del}{\delta}

\newcommand{\antisym}{\mathfrak{O}_n}
\newcommand{\symn}{\text{sym$_n$}}
\newcommand{\sym}{\text{sym}}
\renewcommand{\skew}{\operatorname{skew}}
\newcommand{\psym}{\text{Psym}_n}

\newcommand{\deuc}{d^{\text{Euc}}}
\newcommand{\geuc}{g^{\text{Euc}}}
\newcommand{\dext}{d^{\text{ext}}}
\newcommand{\dint}{d^{\text{int}}}

\numberwithin{equation}{section}

\begin{document}

\title{On strain measures and the geodesic distance to $\SOn $  in the general linear group}
\author{
Raz Kupferman and 
Asaf Shachar \\
\\
Institute of  Mathematics \\ 
The Hebrew University \\
Jerusalem 91904 Israel}
\maketitle
\tableofcontents

\begin{abstract}
We consider various notions of strains---quantitative measures for the deviation of a linear transformation from an isometry. The main approach, which is motivated by physical applications and follows the work of \cite{NEM15}, is to select a Riemannian metric on $\GL$, and use its induced geodesic distance to measure the distance of a linear transformation from the set of isometries. We give a short geometric derivation of the formula for the strain measure for the case where the metric is left-$\GL$-invariant and right-$\On$-invariant.
We proceed to investigate alternative distance functions on $\GL$, and the properties of their induced strain measures. We start by analyzing Euclidean distances, both intrinsic and extrinsic. Next, we prove that there are no bi-invariant distances on $\GL$. Lastly, we investigate strain measures induced by inverse-invariant distances.   
\end{abstract}

 %We also discuss the importance of requiring bi-$\On$-invariance, rather than $\SOn$-invariance.

%%%%%%%%%%%%%%%%%%%%%%%%%%%%%%%%%%%%%%%%%%%%%
\section{Introduction}

In various physical and mathematical contexts, a natural question arises: how to quantify the distortion of an invertible linear transformation $A\in\GL$? 
That is, how far is $A$ from being an isometry? 
In material science, the local distortion of a map between two manifolds is known as a \Emph{strain measure}.
%---a quantity that plays a central role in both statics and dynamics.

One can investigate various notions of strain measures.
A natural approach is to choose a distance function $d$ on $\GL$, and define the strain measure as follows:
\[
\strain{A} = \distSO{A}=\inf_{Q \in \SOn} d(A,Q).
\]
Since $\SOn$ is compact, the distance is realized for some $Q\in\SOn$.

In material science, the configuration of a body is a map $f$ from a body manifold $\mathcal{B}$ to a space manifold $\mathcal{S}$. If both manifolds are endowed with Riemannian metrics, then one can define a local strain measure at every point $p$ of the body manifold,
\[
\strain{df} = \distSO{df},
\]
where $\SOn$ here refers to the space of pointwise orientation-preserving isometries. By choosing  orthonormal frames at both $p$ and $f(p)$, invertible linear maps between tangent spaces can be identified with $\GL$, whence the relevance of the proposed framework to general Riemannian settings.

The notion of strain measure depends on the choice of a distance function $d$.
In physical applications, one expects this distance to satisfy certain symmetries with respect to left- and right-multiplication---the former is related to symmetries of the ambient space whereas the latter is related to material symmetries. The most common symmetry assumptions are frame-indifference, which is left-$\On$-invariance, and material isotropy, which is right-$\On$-invariance.

Left- and right-$\On$-invariance do not determine a unique distance on $\GL$, nor do they determine a unique strain measure. The most common distance $d$ is the so-called Frobenius, or Euclidean distance,
\[
\deuc(A,B) = \|A-B\|_F,
\]
where $\|A\|_F^2 = \tr(A^TA)$. The Euclidean distance results in a strain measure given by
\[
\strainEuc{A} = \|\sqrt{A^TA} - I\|_F.
\]

The Euclidean strain measure suffers from well-known drawbacks. From a physical point of view, the main drawback is that $\strainEuc{A}$ remains finite as $A$ tends toward singularity. The Euclidean strain ``penalizes" extreme expansions, but does not ``penalize" extreme contractions. 

The space $\GL$ is a smooth submanifold of the space $M_n$ of $n\times n$ matrices.
Thus, a natural way to define a distance $d$ on $\GL$ is via a Riemannian metric $g$. In this context, the Euclidean distance $\deuc$ is induced by the Euclidean metric on $M_n$,
\[
\geuc_Z(X,Y) = \tr(X^TY), \qquad  \text{where }  \, X,Y \in T_ZM_n \simeq M_n
\]
For $A,B\in M_n$, $\deuc(A,B)$ is the length of the segment $[A,B]$ with respect to the metric $\geuc$. 

A note about terminology:  
to avoid confusion, we will use the term ``distance", rather than ``metric"  in the context of a metric space. The term ``metric" will be reserved for Riemannian metrics.

From a mathematical point of view, a drawback of $\deuc|_{\GL}$ as a distance function on $\GL$ is that it is not an intrinsic distance. Since $\GL$ is not convex, segments $[A,B]$, $A,B\in\GL$ may not be contained in $\GL$. 

The drawbacks of the Euclidean strain measure are at the heart of a series of papers by Neff and co-workers \cite{neff2013logpolar,LNN14,MN14,NEM15}. They endow $\GL$ with a metric that possesses an additional symmetry: in addition to the bi-$\On$-invariance, they assume left-$\GL$-invariance; this is perhaps the most symmetric choice, as it is well-known that there are no bi-invariant metrics on $\GL$.  This additional symmetry restricts drastically  the set of possible metrics. The left-$\GL$ invariance implies that the metric is fully determined by its value at the identity. The addition of right-$\On$-invariance yields a family of metrics depending only on three parameters.

It was shown in \cite{NEM15} that the unique matrix in $\SOn$ that is the closest to $A \in \GLp$ is its orthogonal polar factor $O$, where $A=OP$, with $O \in \SOn$ and $P \in \psym$. Moreover, a closed formula for the strain measure was derived,
\beq 
\strain{A} = \distSO{A} = \|\log \sqrt{A^TA} \|,
\label{eq:Neff_Strain_Formula}
\eeq

where the logarithm of a symmetric positive-definite matrix is its unique symmetric logarithm, and
the norm $\|\cdot\|$ depends on the three parameters mentioned above (see \eqref{eq:expression_inner_product}  below).
This strain measure diverges in singular limits. In particular, it is \Emph{inverse-invariant}, i.e  
\[
\strain{A} = \strain{A^{-1}}.
\]

In this paper we provide an elementary derivation of formula \eqref{eq:Neff_Strain_Formula} for the strain measure.
Using geometric insights, the set of all possible minimizing paths from a given $A \in \GLp$ to $\SOn$ is narrowed considerably. This helps determining the minimal distance in an elementary way.
In particular, our analysis clarifies the different roles played by the various symmetries of the metrics.

In section \ref{sec:symm_geo}, we introduce the family of left-$\GL$- right-$\On$-invariant metrics. We  state a key property satisfied by these metrics---orthogonality relations---which play a central role in the forthcoming analysis. 
We also describe the form of the geodesics. Section \ref{sec:geo_dist} contains the derivation of the corresponding strain measure. In Subsection~\ref{eq:hemitropy_at_identity}, we shed light on the reasons for assuming $\On$-invariance, rather than $\SOn$-invariance, which might have seemed a more natural assumption.

In \secref{sec:Int_vs_Ext}, we turn to analyze extrinsic versus intrinsic distances, first in a general Riemannian setting and then applied to the case of $\GL$ viewed as a submanifold of $M_n$ endowed with the Euclidean metric. The main result is that while the intrinsic distance differs from the extrinsic distance, the strain measures are the same in both cases.
In particular, we give a very short derivation of Grioli's optimality theorem \cite{grioli1940} (see also \cite{neff2013grioli}), which says that for a given $A \in \GLp$, its orthogonal polar factor is the closest matrix to $A$ in $\SOn$ with respect to the Frobenius norm.

In Section~\ref{sec:invariant_distances}, we investigate how other natural symmetries on distance functions affect the strain measure. We start by showing there are no bi-invariant distance functions on $\GL$, hence there is an ``upper limit'' to the amount of symmetries a distance function can possess (see Subsection~\ref{subsec:bi_invariance}).

Next, we show that an inverse-invariant strain measure is obtained if the distance/metric is inverse-invariant. We then describe two different techniques for obtaining such distances/metrics via symmetrization, and analyze the resulting strain measures.
In the case of symmetrizing a distance, we investigate the two families of distances considered thus far: the Euclidean (intrinsic and extrinsic) distance, and the (intrinsic) distances induced by the metrics considered in Section~\ref{sec:symm_geo}. In the case of the Euclidean distance, the result is an improved strain measure, which penalizes expansions and contractions equally. In the other cases, the strain measure is essentially the same as without the symmetrization.

Finally, we discuss the symmetrizations of all the metrics considered in Section~\ref{sec:symm_geo}. The resulting strain measure is also essentially the same as the original. The proof contains an analysis of metrics that are expressed as sums of two metrics,  and also sheds light on the key ingredients in the derivation of the strain measure in Subsection~\ref{subsec:geodesic_dist_diagonal}. 

%%%%%%%%%%%

%%%%%%%%%%%%%%%%%%%%%%%%%%%%%%%
\section{Symmetries and geodesics}
\label{sec:symm_geo}

%%%%%%%%%%%%%%%%%%%%%%%%%%%%%%%
\subsection{Left-$\GL$- and right-$\On$-invariant metrics}

Throughout this paper, we use the following notations:

$\GL$ is the group of $n \times n$ invertible real matrices, $\GLp$ and $\GLm$ are the connected components of $\GL$, i.e., $\GLp$ is the subgroup of $n \times n$  invertible matrices with positive determinant, and $\GLm$ is the subset of matrices with negative determinant. We denote by
\[
\On=\{Q \in \GL ~|~ Q^TQ=I \}\subset \GL
\] 
the subgroup of orthogonal matrices, whereas $\SOn\subset \GLp$ is the subgroup of special orthogonal matrices, i.e those with determinant $1$.

We will denote by $M_n$ the vector space of $n \times n$ real matrices, and by $\psym\subset M_n$ the cone of symmetric positive-definite matrices.

For readability, we will try to stick to the following choice of symbols:
\[
\begin{aligned}
& A,B \in \GL \\
& O,U,V \in \On \\
& Q \in \SOn \\
& X,Y \in M_n \\
& P \in \psym.
\end{aligned}
\]

Let $g$ be a left-$\GL$- and right-$\On$-invariant metric on $\GL$. 
A left-invariant metric $g$ on a Lie group $G$ is determined by its restriction at the identity. For $A\in \GL$, let $L_A:\GL \to \GL$ denote left multiplication by $A$, i.e $L_A(B)=AB$. $L_A$ is a diffeomorphism and its differential $(dL_A)_I:T_I\GL \to T_A\GL$ is a vector space isomorphism. For all $X,Y\in T_I\GL$,
\beq
%g_A(X,Y) = g_I(A^{-1}X,A^{-1}Y).
g_I(X,Y) = g_A\brk{(dL_A)_IX,(dL_A)_{I}Y}.
\label{eq:invariance}
\eeq

Since $\GL$ is an open subset of $M_n$, its tangent space at each point is canonically identified with $M_n$ as follows: Given $A \in \GL$, the identification $i_A:M_n \to T_A\GL$ is $i_A(X)=[t \mapsto A+tX]$.

The action of the differential $dL_A$ on a tangent vector at $B$ is 
\[
(dL_A)_Bi_B(X) = (dL_A)_B([B+tX]) = [AB+tAX]  = \iota_{AB}(AX).
\]
Using the above identification,
\[
(dL_A)_B X = AX,
\]
where the dependence of the right-hand side on $B$ is implicit via the identification of $M_n$ with $T_{AB}\GL$.

Substituting this last identify for $B=I$ into \eqref{eq:invariance} we obtain that left-$\GL$-invariance implies, 
\[ 
g_I(X,Y)=g_A(AX,AY) \qquad \forall A \in \GL.
\]

Similarly, right-$\On$-invariance implies
\[ 
g_A(X,Y)=g_{AO}(XO,YO)  \qquad \forall O \in \On.
\]

An immediate consequence of both left- and right-$\On$-invariance, is that $g_I$ is isotropic.
For every $U \in \On$:
\beq
g_I(X,Y) = g_{U^T}(U^TX,U^TY)=g_{U^TU}(U^TXU,U^TYU) = g_I(U^TXU,U^TYU).
\label{eq:isotropy_at_identity}
\eeq
In fact, the same argument shows that for any Lie group $G$ and subgroup $H\subseteq G$, a left-invariant metric $g$ is right-$H$-invariant if and only if  $g_e$ is invariant under conjugation with elements in $H$.

From a representation theorem for isotropic operators ~\cite{boor1985naive}, it follows that there exist constants $\al,\beta\ge0$ and $\ga\le0$, such that
\beq
g_I(X,Y)=\al\tr(X)\tr(Y) + \beta\tr(\sym X\, \sym Y) + \ga\tr(\skew X\,  \skew Y),
\label{eq:expression_inner_product}
\eeq
where $\sym X$ and $\skew X$ denote respectively the symmetric and skew-symmetric parts of $X$. 
%In particular, the norm of a diagonal matrix, $\|D \|_I=\sqrt{g_I(D,D)}$, is invariant under permutations of its diagonal elements; for any permutation $\tau$,
%\be%q
%\|\%diag(d_{\tau(1)},\dots,d_{\tau(n)})\|_I= \|\diag(d_1,\dots,d_n) \|_I.
%\la%bel{eq:diagonal_norm_invariance}
%\ee%q
%Thi%s follows from the isotropy \eqref{eq:isotropy_at_identity}, taking $S$ to be a permutation matrix.
%

Let $\sym\subset M_n$ and $\antisym\subset M_n$ denote the subspaces of symmetric and anti-symmetric matrices in $M_n \cong T_I\GL$. The following lemma asserts that these sets are orthogonally complementary with respect to $g_I$:

%%%%%%%%%
\begin{lemma}
\label{lem:symmetric_antisym_orthogonal}
Let $g_I$ satisfy the isotropy condition \eqref{eq:isotropy_at_identity}. Then, $\sym$ and $\antisym$ are orthogonally complementary.
\end{lemma}
%%%%%%%%%

%%%%%%%%d
\begin{proof}
The orthogonality of $\sym$ and $\antisym$ can be shown by an explicit substitution in the form \eqref{eq:expression_inner_product} of the metric, hence $\sym\subseteq\antisym^\perp$. The fact the these subspaces are complementary follows from a dimensional argument,
\[
\dim\sym + \dim\antisym = n^2 = \dim T_I\GL.
\]

\end{proof}
%%%%%%%%%

%%%%%%%%%%%%%%%%%%%%%%%%%%%%%%%
\subsection{Geodesics}

In this section we review the properties of geodesic curves in $(\GL,g)$. 

%%%%%%%%%%%%
\begin{proposition}[$g$-geodesics starting at the identity]
\label{prop:geodesics_expression}
Let $g$ be left-$\GL$, right-$\On$-invariant. Let $g_I$ be given by \eqref{eq:expression_inner_product} and denote $\kappa = (\beta - \ga)/2\beta$. 
Let $\gamma:I\to \GL$ be the $g$-geodesic, satisfying the initial conditions
\[
\gamma(0) = I
\Textand
\dot{\gamma}(0) = X_0.
\]
Then,
\[
\gamma(t) =  \exp((1-\kappa)t X_0 + \kappa t X_0^T) \,\exp(\kappa t (X_0-X_0^T)).
\]
\end{proposition}
%%%%%%%%%%%%

%%%%%%%%%%%%
\begin{proof}
This was proved in \cite{MN14} using an argument based on variations of energy. 
A shorter alternative proof using Cartan's moving frame method is given in Appendix~\ref{sec:geosesic_eqs}.
\end{proof}
%%%%%%%%%%%%

%%%%%%%%%%%%
\begin{corollary}[$g$-geodesics]
\label{prop:general_geodesics_expression}
Under the same assumptions as above, let $\gamma:I\to \GL$ be the $g$-geodesic satisfying the initial conditions
\[
\gamma(0) = A 
\Textand
\dot{\gamma}(0) = AX_0.
\]
Then,
\[
\gamma(t) =  A \exp((1-\kappa)t X_0 + \kappa t X_0^T) \,\exp(\kappa t (X_0-X_0^T)).
\]
\end{corollary}
%%%%%%%%%%%%

%%%%%%%%%%%%
\begin{proof}
This follows from the fact that left multiplication is an isometry of $(\GL,g)$. It is a general property of Riemannian manifolds that isometries map geodesics into geodesics.
\end{proof}
%%%%%%%%%%%%

%%%%%%%%%%%%

%%%%%%%%%%%%
\begin{corollary}
\label{cor:geodesic_sym_velocity2}
 Let $\gamma:I\to \GL$ be the $g$-geodesic satisfying the initial conditions
\[
\gamma(0) = Q
\Textand
\dot{\gamma}(0) = QV,
\]
where $Q\in\On$ and $V\in\sym$.
Then,
\[
\gamma(t) =  Q\exp(tV) .
\]
\end{corollary}
%%%%%%%%%%%%

%%%%%%%%%%%%%%%%%%%%%%%%%%%%%%%%%%%%%
\section{Geodesic distance from $\SOn$}
\label{sec:geo_dist}

Every Lie group endowed with a left-invariant metric is complete as a Riemannian manifold. That is, every geodesic extends indefinitely. This follows from the fact that its isometry group acts transitively; see \cite[p. 154, Example~12]{DOC92}. By the Hopf-Rinow theorem, \cite[p. 146]{DOC92} the length-distance between any two points is realized by a minimizing geodesic.

Generally, there doesn't seem to exist any explicit expression for the (possibly many) geodesics connecting any two elements $A,B\in\GLp$, nor for the resulting distance between these elements. Yet, we are only interested in the distance of an element $A\in\GLp$ from the subgroup $\SOn$ of isometries.  As demonstrated in \cite{NEM15}, an explicit expression can be derived for that distance. In this section we offer a simplified derivation of that expression.

%%%%%%%%%%%%%%%%%%%%%%%%%%%%%%%%%%%%%
\subsection{Reduction to diagonal positive-definite matrices}

The first step in calculating the distance of $A\in\GLp$ from $\SOn$ is to show that it is sufficient to obtain a formula for diagonal matrices.
The following proposition holds for any bi-$\SOn$-invariant distance on $\GLp$---not necessarily a distance induced by a Riemannian metric.

%%%%%%%%%%%%%
\begin{proposition}
\label{prop:orthogonal_invariance_dist}
Let $d$ be a  bi-$\SOn$-invariant distance on $\GLp$; we denote the corresponding distance between sets by $\dist$.
Let $A \in\GLp$. If $A= U \Sigma V^T$ is a singular value decomposition (SVD) of $A$ with $U,V \in \SOn$, then 
\[
\distSO{A}=\distSO{\Sig}.
\]
Moreover, if  
$Q$ is a matrix closest to $\Sig$ in $\SOn$, then $UQV^T$ is a matrix closest to $A$ in $\SOn$. 
\end{proposition} 
%%%%%%%%%%

\begin{proof}
We first note that for every $A \in \GLp$, there exists an SVD such that $U,V \in \SOn$ (see the comment after the proof of \corref{cor:Expression_for_the_Frobenius_minimizer}). Moreover,  $\Sig$ is unique (up to permutation), i.e., the singular values do not depend on the particular decomposition.

Assuming $U,V \in \SOn$ and using the bi-$\SOn$-invariance,
\beq
\label{eq:orthogonal_distance_invariance}
\begin{split}
\distSO{A} &=
\min_{Q \in\SOn} d(A,Q) =\min_{Q \in\SOn} d(U \Sig V^T,Q) \\
&=  
%\min_{Q \in\SOn} d(\Sig V^T,U^TQ) =  
\min_{Q \in\SOn} d(\Sig, U^TQV)  = \distSO{\Sig}.
\end{split}
\eeq
The last equality holds since $\{U^T Q V~|~Q\in\SOn\} = \SOn$.
Equation \eqref{eq:orthogonal_distance_invariance} implies that $Q \in \SOn$ is a matrix closest to $\Sig$ in $\SOn$ if and only if $UQV^T$ is a matrix closest to $A$ in $\SOn$.
\end{proof}
%%%%%%%%%

%%%%%%%%%%%%%%%%%%%%%%%%%%%%%%
\subsection{Geodesic distance for diagonal matrices}
\label{subsec:geodesic_dist_diagonal}

By \propref{prop:orthogonal_invariance_dist}, we can focus our attention on finding the distance from $\SOn$ for diagonal positive-definite matrices, $\Sig$. Since $\GLp$ is complete, we look for a minimizing geodesic from $\Sig$ to $\SOn$.
To do so, we are going to exploit the fact that any geodesic minimizing the distance of a point to a submanifold  intersects that submanifold perpendicularly. More precisely:

%%%%%%
\begin{lemma}
\label{lem:minimizing_geodesic_orthogonal} 
Let $M$ be a complete Riemannian manifold. Let $S \subseteq M$ be a submanifold, and let $p \in M \setminus S$. Assume $q \in S$ is a point on $S$ satisfying $d(p,q)=\dist(p,S)$ (there is always such a point $q$ if $S$ is compact). Let $\al$ be a minimizing geodesic connecting $p$ and $q$. Then $\al$ is orthogonal to $S$ at $q$.
\end{lemma}
%%%%%%

See \cite{243767} for a proof.

%%%%%%%%%
\begin{proposition}
Let $\Sig=\diag(\sig_1,\dots\sig_n)$ be a diagonal matrix with positive entries. Then, 
\[
\distSO{\Sig} = d(\Sig,I).
\]
Moreover, $I$ is the unique element of $\SOn$ minimizing the distance from $\Sig$.
\end{proposition}
%%%%%%%%

%%%%%%%%
\begin{proof}
Let $Q \in \SOn$ satisfy
\[
d(\Sig,Q)=\distSO{\Sig}.
\]
By the completeness of $\GLp$, there exists a minimizing geodesic
$\al:[0,1] \to \GLp$ from $Q$ to $\Sig$, i.e., 
\[ 
\al(0)=Q, \quad 
\al(1)=\Sig
\textand 
L(\al)=\|\dot\al(0)\|_{Q}=d(\Sig,Q). 
\]
\lemref{lem:minimizing_geodesic_orthogonal} implies that $\dot \al(0) \perp T_Q\SOn$. 

Denoting $\dot\al(0)=QV$ (where we think of $V$ as an element of $T_I\GLp \cong M_n$, and $QV$ is identified with $d(L_Q)_I(V)$), we obtain for any antisymmetric matrix $X \in \antisym=T_I\SOn$:
\[
g_I(X,V) = g_Q(QX,QV) = g_Q(QX,\dot\al(0)) = 0,
\]
where the last equality is valid since $d(L_Q)_I(T_I\SOn)=T_Q\SOn$, hence $QX=d(L_Q)_I(X) \in T_Q\SOn$.

Thus, $V\in T_I\GLp$ is orthogonal to every anti-symmetric matrix, and by  \lemref{lem:symmetric_antisym_orthogonal}, $V\in\sym$. 
It follows from \corref{cor:geodesic_sym_velocity2} that
\[
\al(t)=Qe^{tV}. 
\]
By the definition of $\al$, $\al(1)=\Sig=Qe^V $. Since $V$ is symmetric, $e^V$ is symmetric positive-definite, hence we obtain two polar decompositions of $\Sig$,
\[
\Sig = I \Sig 
\Textand
\Sig = Q  e^V. 
\]
By the uniqueness of polar decomposition for invertible matrices, we conclude that $Q=I$, which completes the proof.
\end{proof}
%%%%%%%%

We proceed to derive an explicit formula for the distance of $\Sig$ from $\SOn$. 
Since $Q = I$, 
\[
e^V = \Sig = e^{\log \Sig},
\] 
where $\log \Sig = \diag(\log \sig_i)$. Since $V$ and $\log \Sig$ are symmetric, and since the matrix exponential is injective on the space of symmetric matrices, it follows that  $V = \log \Sig$.
Hence,
\[
\distSO{\Sig}=  \|\dot{\al}(0)\|_I = \|V\|_I=\| \log \Sig\|_I. 
\]

%Since the minimizing geodesic connecting $I$ to $\Sig$ is $\al(t) = \exp tV$, where $V\in\sym$ and $\Sig = e^V$, it follows that 
%
%\[
%\distSO{\Sig} = \|\dot{\al}(0)\|_I = \|V\|_I.
%\]
%
%Since $V$ is symmetric it can be diagonalized by an orthogonal transformation, 
%\[
%V=ODO^T,
%\]
%where $O \in \On$ and $D=\diag(\lam_1,\dots\lam_n)$. It follows that
%\[
%\Sig=e^V=Oe^DO^T,  
%\]
%which implies that the singular values of $\Sig$ and $Oe^DO^T$ are equal.  Since  singular values are unaffected by left or right multiplication by orthogonal matrices, the singular values of $\Sig$ and $e^D$ are equal. The singular values of a diagonal positive-definite matrix are its diagonal entries, hence 
%\beq
%\lam_i = \log \sig_{\tau(i)}
%\label{eq:lam_log_sig}
%\eeq
%for some permutation $\tau$.
%
%Putting it all together,
%\beq
%\begin{split}
%  \distSO{\Sig} &= d(\Sig,I)= \|V \|_I = \|ODO^T \|_I \\
%& = \|D\|_I  \\
%&= \|\diag(\log\sig_{\tau(1)},\dots,\ln\sig_{\tau(n)})\|_I \\
%&= \|\diag(\log\sig_1,\dots,\ln\sig_n) \|_I \\
%&= \|\log\Sig \|_I
%\end{split}
%\label{eq:dist_Sig_SO}
%\eeq
%where the passage to the second line follows from the isotropy \eqref{eq:isotropy_at_identity} of $g_I$, the passage to the third line follows from \eqref{eq:lam_log_sig}, and the passage to the fourth line follows from the invariance \eqref{eq:diagonal_norm_invariance} of the norm of a diagonal metric under a permutation of its diagonal entries. The passage to the lest line is just by definition.
%

Substituting the explicit form \eqref{eq:expression_inner_product} of the metric $g_I$,
\beq
\distSO{\Sig}= \sqrt{\al\brk{\sum \log\sig_i}^2+\beta \sum \brk{\log \sig_i}^2}.
\label{eq:explicit_dist_Sig_SO}
\eeq
As a corollary, we get that $\al(t)=e^{t\log \Sig}$ is the unique minimizing geodesic connecting $I$ to $\Sig$.

%%%%%%%%%%%%%%%%%%%%%%%%%%%%%%

\subsection{Geodesic distance for arbitrary matrices}

Let $A \in \GLp$ be an arbitrary matrix.
If $A=U\Sig V^T $ is an SVD of $A$, then $\sqrt{A^TA}=V\Sig V^T$, hence $\log\sqrt{A^TA}=V \, \log \Sig \, V^T$.
By \propref{prop:orthogonal_invariance_dist},
\beq
\label{eq:explicit_dist_general_matrix_SO} 
\distSO{A}=\| \log \Sig\|_I = \| V \, \log \Sig \, V^T\|_I=\|\log\sqrt{A^TA}\|_I,
\eeq
where the second equality follows from the invariance \eqref{eq:isotropy_at_identity}. Whenever we write $\log$ of a symmetric positive-definite matrix, we refer to its unique symmetric logarithm. Since the exponential map is a diffeomorphism from $\symn$ to $\psym$ there is no ambiguity here.
We have thus obtained an explicit expression for the distance of any matrix $A\in\GLp$ from $\SOn$ by elementary means.

We have shown that for a diagonal positive-definite matrix $\Sig$, $Q=I$ is the unique element in $\SOn$ satisfying $\distSO{\Sig} = d(\Sig,Q)$. By \propref{prop:orthogonal_invariance_dist}, if $A = U\Sig V^T$ is an SVD of $A\in\GLp$ with $U,V\in\SOn$, then
$UV^T$ is the unique matrix in $\SOn$ that is closest to $A$.

Moreover:

%%%%%%%%%%
\begin{corollary}[The orthogonal polar factor is the minimizer]
\label{cor:Expression_for_the_Frobenius_minimizer}
Let  $A \in \GLp$.
Let $A=OP$ be the polar decomposition of $A$, $O \in \SOn$ and $P \in \psym$. 
Then $O$ is the matrix closest to $A$ in $\SOn$. 
\end{corollary}

\begin{proof}
By orthogonally diagonalizing $P$ with $P = \til U \Sigma \til U^T$, we obtain an SVD, 
\[
A = O\til U\Sigma \til U^T=U \Sigma V^T,
\] 
where $U=O\til U$ and $V=\til U$. 
Note that by interchanging two columns if necessary, we can assume $\til U \in \SOn$, hence $V,U \in \SOn$.
By the above discussion,  $UV^T=O\til U \til U^T = O$ is the matrix closest to $A$.
\end{proof}

Please note: the above argument shows that for $A \in \GLp$ there always exists an SVD where both orthogonal matrices are in $\SOn$.

%%%%%%%%%%%%%%%%%%%%%%%%%%%%%%%%%%%%%
\subsection{$\On$ versus $\SOn$-invariance}

The analysis presented in Sections~\ref{sec:symm_geo} and \ref{sec:geo_dist} assumes that the metric $g$ is left-$\GL$- and right-$\On$ invariant.
Since we are interested in intrinsic distances in $\GLp$ from the subgroup $\SOn$, it may seem as if we could perform the whole analysis in $\GLp$ rather than in $\GL$. In such case, it only makes sense to require the Riemannian metric to be left-$\GLp$ and right-$\SOn$ invariant---right $\On$-invariance, for example, is meaningless. A natural question is the following: would we obtain the same geodesic distances and the same strain measures if we considered left-$\GLp$ and right-$\SOn$-invariant metric on $\GLp$?

An inner-product $g_I$ satisfying condition \eqref{eq:isotropy_at_identity} is called \textbf{isotropic}. In contrast, an inner-product $g_I$ satisfying
\beq
g_I(X,Y)=g_I(S^TXS,S^TYS), \qquad \forall S \in \SOn
\label{eq:hemitropy_at_identity}
\eeq
is called \textbf{hemitropic}.
If every hemitropic inner-product is isotropic, then our entire analysis extends as is to $\SOn$-invariant metrics on $\GLp$. 
If, however, isotropy and hemitropy are not equivalent, then our analysis has to be revisited, as 
the representation of the inner-product \eqref{eq:expression_inner_product} relies explicitly on the isotropic nature of the inner-product $g_I$.

Our analysis relies on the specific form  \eqref{eq:expression_inner_product} of the inner-product $g_I$ in two crucial aspects: (i) in the derivation of an explicit formula for the geodesics, and (ii) in obtaining the orthogonality of symmetric and anti-symmetric matrices. Since an inner-product is of the form \eqref{eq:expression_inner_product} if and only if it is isotropic, any hemitropic, but non-isotropic inner-product is not of that form, hence our analysis is not applicable.

It turns out that for all dimensions $n \neq 4$, there are no hemitropic non-isotropic inner-products. For odd $n$ this is trivial to see since $-I \in \On \setminus \SOn$ commutes with every other matrix. The analysis for even dimensions is less trivial. A proof can be found in ~\cite{229549} .
Thus, our work holds as is with isotropy replaced by hemitropy in any dimension other than $4$. 
Understanding the implications of an hemitropy assumption for $n=4$ remains an open question.

%%%%%%%%%%%%%%%%%%%%%%%%%%%%%%%%%%%%%%
\section{Intrinsic versus extrinsic distances}
\label{sec:Int_vs_Ext}

Endowing $\GLp$ with distances induced by Riemannian metrics is one type of choice for 
measuring the distortion of a linear map. Another popular choice is the distance induced by the Frobenius inner-product on $M_n$, or equivalently, the Euclidean metric on $M_n$ identified with $\R^{n^2}$. In fact, one of the motivations in \cite{NEM15} for considering distances induced by Riemannian metrics was the claimed inadequacy of the Euclidean metric. Note that the Euclidean metric gives rise to two distinct distances on $\GLp$: (i) an \Emph{extrinsic} distance, obtained by restricting the Euclidean distance function to the subset $\GLp$ of $M_n$, and (ii) an \Emph{intrinsic} length-distance determined by paths in $\GLp$.

In this section we explore the distinction between extrinsic and intrinsic distances, first in a general Riemannian context, and second, in the original Euclidean context.

%%%%%%%%%%%%%%%%%%%%%%%%%%%%%%%%%%%%%%
\subsection{The general  Riemannian case}

Let $(M,g)$ be a Riemannian manifold. Denote the induced Riemannian distance function by $d^M$. Let $S \subset M$ be an embedded connected submanifold.

There are two natural ways to induce a distance on $S$:
\begin{enumerate}
\item Intrinsic: Consider $S$ as a Riemannian submanifold of $M$, i.e., endow $S$ with the pullback metric $i^*g$ along the inclusion $i:S\to M$. Denote by $\dint_S$ the Riemannian distance function induced by $i^*g$.

\item Extrinsic: Consider $S$ as a subspace of the metric space $(M,d^M)$. Denote by  $\dext_S$  the restriction of $d^M$ to $S \times S$.
\end{enumerate}

An immediate observation is that $\dint_S \ge \dext_S$. The question we pose is under what conditions, $\dint_S = \dext_S$.

In general, both equality and inequality may hold: 
For $M=\mathbb{S}^2$ endowed with the round metric and $S$ a great circle, $\dint_S = \dext_S$.
For $M=\mathbb{R}^2$ with the standard Euclidean metric and $S=\mathbb{S}^1$, $\dint_S > \dext_S$.

To state our results we need the following classical definitions:

%%%%%%%%%%
\begin{definition}
Let $(M,g)$ be a Riemannian manifold.
A subset $C$ of $M$ is said to be a \Emph{geodesically convex} if, given any two points in $C$, there is a minimizing geodesic (in $M$)  contained within $C$ joining these two points.
\end{definition}
%%%%%%%%%%

%%%%%%%%%%
\begin{definition}
A submanifold $S$ of a Riemannian manifold $(M,g)$ is called \Emph{totally geodesic} if any geodesic on the submanifold $S$ with its induced Riemannian metric is also a geodesic on the Riemannian manifold $(M,g)$. (This condition is equivalent to the vanishing of the second fundamental form of $S$ in $M$.)
\end{definition}
%%%%%%%%%%

To prove our results we shall need the following lemma which roughly says that  paths that are close to being length-minimizers are within a narrow tubular neighborhood of a (minimizing) geodesic. 
 
%%%%%%%%%%%  
\begin{lemma}[Nearly length-minimizing paths are close to geodesics]
\label{lem:nearly_length_minimizing_close_geodesics}
Let $(M,g)$ be a complete Riemannian manifold, and let $p,q \in M$. 
Then, for any $\epsilon>0$ there exists a $\delta>0$ (possibly dependent on $p$ and $q$) such that any path $\alpha$ joining $p$ and $q$ satisfying
\[
L(\alpha) < d(p,q) + \delta
\]

is in an $\epsilon$-neighborhood of a minimizing geodesic $\gamma:I\to M$ joining $p$ and $q$.
In particular, there exists a reparametrization $\alpha \circ \varphi :I \to M$ of $\alpha$ satisfying
\[
\sup_{t \in I} d\brk{\alpha \circ \varphi(t),\gamma(t)} < \epsilon.
\]
\end{lemma}
%%%%%%%%%%%  

In this lemma there is no submanifold, so the distance has only one possible meaning--- the Riemannian distance on $M$.

%%%%%%%%%%%  
\begin{proof}
Assume by contradiction that the claim is false. Denote $r=d(p,q)$.
Then, there exists an $\epsilon > 0$ and a sequence of paths $\alpha_n:I \to M$ joining $p$ and $q$, satisfying 
\[
L(\alpha_n) \le d(p,q) + \frac{1}{n}, 
\]
and $\alpha_n$ is not in an $\epsilon$-neighborhood of any minimizing geodesic.

Since $L(\alpha_n)\to r$, we can assume $L(\al_n) \le 2r$, thus $\operatorname{Image}(\alpha_n) \subseteq \bar  B^M(p,2r)$ (the closed ball of radius $2r$ around $p$). By the completeness of $M$, it follows from the Hopf-Rinow theorem that $\bar B^M(p,2r)$ is compact.

Reparametrize $\al_1$ by arclength, i.e assume $I=[0,L(\al_1)]$, and that $\al_1:I \to M$ has a constant speed. For every $n \in \mathbb{N}$,
reparametrize $\al_n:I \to M$ such that it has a constant speed $c_n=\| \dot \al_n \|$. Then 
\[ 
2L(\al_1) \ge 2r \ge L(\al_n)=c_nL(\al_1) \quad\Rightarrow\quad c_n \le 2,
\] 
which implies that the $\al_n$ are equicontinuous, since \[d(\al(t),\al(s))\le L(\al_{[t,s]})\le 2(s-t). \] 

By the Arzela-Ascoli theorem, there exists a subsequence (also denoted $\alpha_n$) converging uniformly to a path $\alpha:I \to M$.
By the lower-semicontinuity of the length functional we deduce:
\[ 
L(\alpha) \le \lim_{n \to \infty} L(\alpha_n) = d(p,q).
\]
This implies that $\alpha$ is a length-minimizing curve between $p$ and $q$,  hence its reparametrization by arclength $\alpha \circ \varphi$  is a geodesic.

Finally, the uniform convergence $\alpha_n \to \alpha$ yields a contradiction: there exists an $N$ such that for all $n>N$,
\[
  \sup_{t\in I} d\brk{(\alpha_n \circ \varphi)(t),(\alpha \circ \varphi)(t)} < \epsilon.
\]

\end{proof}

We next prove the following:

%%%%%%%%%
\begin{proposition}
\label{prop:riemannian_intrinsic}
Let $S$ be a submanifold $S$ of a Riemannian manifold $(M,g)$. Then:
\begin{enumerate}
\item If $S$ is a geodiscally convex subset of $M$, then $\dint_S = \dext_S$. The reverse implication does not hold in general.
The next assertion shows that the only obstruction for the reverse direction to hold, is topological.
\item If $S$ is topologically closed in $M$, then $\dint_S = \dext_S$ if and only if  $S$ is a geodesically convex subset of $M$.
\item If $\dint_S = \dext_S$ then $S$ is a totally geodesic submanifold of  $M$. The reverse implication does not hold in general.
\item   Let $p,q \in S$. Assume there exists a \textbf{unique} minimizing geodesic $\gamma:I\to M$ connecting  $p$ and $q$. If $\gamma \cap (M\setminus \bar S) \neq \emptyset$ , Then $\dint_S(p,q) > \dext_S(p,q)$.

%\item 
%If $M$ is \textbf{uniquely geodesic}, i.e., there exists a unique minimizing geodesic between any two points in $M$, then $\dint_S = \dext_S$ if and only if  $S$ is a geodsically convex subset of $M$.
\end{enumerate}
\end{proposition}
%%%%%%%%

The last two statements hold also if we replace the existence of a unique geodesic with the existence of finitely many geodesics.

%\begin{enumerate}
%\item 
%\[
%\begin{split}
% S & \text{ is geodesically convex subset of } M  \Rightarrow \\
% & d^{Intrinsic}_S = d^{External}_S  \Rightarrow  S \text{ is totally geodesic submanifold of } M,
%\end{split}
%\]
%
%and each of the above implications does not always hold in the opposite direction.
%
%\item
%
%Let $S$ be closed (topologically)  in $M$. Then:
%
%$S$ is a geodesically convex subset of $M \iff d^{Int}_S = d^{Ext}_S $ 
%
%\item
%
%  Let $p,q \in S$. Assume there exists a \textbf{unique} minimizing geodesic $\gamma$ in $M$ between $p,q$. Assume $\gamma \cap (M\setminus \bar S) \neq \emptyset$. Then $d^{Int}_S(p,q) > d^{Ext}_S(p,q)$.
%
%Note: If you put $M=M_n=\R^{n^2},S=\GLp$ , then you are in the case that interested us. There are $A,B \in \GLp$ such that the segment $[A,B]$ passes through $\GLm=M \setminus \bar S$
%
%\item 
%  Assume $M$ is \textbf{uniquely geodesic}, i.e there exists a unique minimizing geodesic between any two points in $M$. Then: 
%
% $d^{Int}_S = d^{Ext}_S \Rightarrow \bar S$ is a geodesically convex subset of $M$
%
%
% ( The last two statements hold also when replacing the assumption on \textbf{unique} geodesic, with \textbf{finite number} of geodesics).
%  \end{enumerate}

%%%%%%%
\begin{proof}
\begin{enumerate}
\item The fact that geodesic convexity implies equality of the distances is immediate. A counter-example for the reverse implication is $M=\R^2,S=\R^2 \setminus {(0,0)}$.

\item 
Assume $d^{Int}_S = d^{Ext}_S$. Let $p,q \in S$. Let $\alpha_n:I \to S$ be a sequence of paths satisfying $L(\alpha_n) \to d(p,q)$. Then, by a similar argument to the one in the proof of \lemref{lem:nearly_length_minimizing_close_geodesics}, there is a subsequence $\alpha_n$ converging uniformly to a path $\alpha$. By the lower-semicontinuity of the length, $L(\alpha)=d(p,q)$, so $\alpha$ is minimizing, hence it is a reparametrization of a geodesic. 
Closedness of $S$ implies $\operatorname{Image}(\alpha) \subseteq S$.
\item 
Suppose that $\dint_S = \dext_S$, and let $\alpha$ be a geodesic in $S$. Then, it is locally length-minimizing in $S$, and for small enough $t$,  
\[
L(\alpha|_{[0,t]})=\dint_S(\alpha(0),\alpha(t))=\dext_S(\alpha(0),\alpha(t))=d^M(\alpha(0),\alpha(t)). 
\]
So $\alpha|_{[0,t]}$ is length-minimizing path between $\alpha(0),\alpha(t)$ in $M$ with parameter proportional to arc length, hence it is a geodesic in $M$. 

A counter-example for the reverse implication is $M=\mathbb{S}^1,S=\mathbb{S}^1 \setminus \{p\}$ (where $p$ is an arbitrary point in $\mathbb{S}^1$). We will see another counter-example in the next section: $M=M_n$ and $S=\GLp$.

\item By assumption, there exists $t_0 \in I$ such that $\gamma(t_0) \in M\setminus \bar S$. Since $M\setminus \bar S$ is open, there is some open ball of $d^M$-radius $\epsilon$, $\gamma(t_0) \in B_{\epsilon} \subset M\setminus \bar S$.

By \lemref{lem:nearly_length_minimizing_close_geodesics}, $\exists \delta > 0$ such that if $\alpha$ is a path between $p,q$, $L(\alpha) < d^M(p,q) + \delta$ then $\alpha$ is in an $\epsilon$-neighborhood of some minimizing geodesic joining $p$ and $q$. By our assumption, there is only one minimizing geodesic between $p$ and $q$ in $M$, namely $\gamma$.

Thus, there exists a reparametrization of $\alpha$, $\alpha \circ \varphi :I \to M$, such that for every $t$,
\[
d\big((\alpha \circ \varphi)(t),\gamma(t)\big) < \epsilon.
\]
In particular, $d\big((\alpha \circ \varphi)(t_0),\gamma(t_0)\big) < \epsilon$ implies that $(\alpha \circ \varphi)(t_0) \in B_{\epsilon} \subseteq M\setminus \bar S$.

This shows that any path $\alpha$ which is $\delta$-close to being a minimizer intersect $M\setminus \bar S$. Hence $d^{Int}_S(p,q) \ge d^{Ext}_S(p,q)  + \delta$.

%\item This is an immediate corollary of 4.
\end{enumerate}
\end{proof}
%%%%%%

%%%%%%%%%

%%%%%%%%%%%%%%%%%%%%%%%%%%%%
\subsection{Euclidean distances in $\GLp$}
\label{subsec:intrinsic_euc}

Next, we consider the particular case where the Riemannian manifold is the vector space of $n\times n$ matrices endowed with the Euclidean metric. This is the case considered classically in the context of elastic strain measures, and whose shortcomings has motivated, in part, the consideration of alternative measures of strain.

We start with a few definitions:

%%%%%%%%%
\begin{definition}[Euclidean metric on $M_n$]
We denote by $(M_n,\deuc)$ the space of $n\times n$ real matrices endowed with the Euclidean distance. Note that the distance $\deuc$ can be derived from a Riemannian metric $g$ given by
\[
g_Z(X,Y) = \tr(X^TY).
\]
\end{definition}
%%%%%%%%

%%%%%%%%%
\begin{definition}[Extrinsic Metric on $\GLp$]
We denote by $(\GLp,\dext)$ the metric space of $n\times n$ invertible matrices with positive determinant, 
where $\dext$ is the restriction of $\deuc$, with $\GLp$ viewed as a subset of $\R^{n^2}$.
%
%  $\operatorname{Mat}(n \times n, \R) \cong \R^{n^2}$ is a metric space with the Euclidean metric $d^{Euc}$. Consider $\GLp$ as a subspace of the metric space $(\R^{n^2},d^{Euc})$. That is, take $\dext$ to be the restriction of $d^{Euc}$ to $\GLp \times \GLp$.
\end{definition}
%%%%%%%%%

%%%%%%%%%
\begin{definition}[Intrinsic Metric on $\GLp$]
Consider  $ \, (\GLp,g|_{\GLp}) \,$ as an open sub-manifold of the Riemannian manifold $(M_n,g)$.
That is, we endow $\GLp$ with the pullback metric $i^* g$ of the Euclidean metric $g$ along the inclusion $ i:\GLp \to (M_n,g)$.
%
%with the induced Riemannian metric from $\mathbb{R}^{n^2}$. (So it's a Riemannian submanifold of $\mathbb{R}^{n^2}$).
%    For clarity, this means endowing $\GL$ with the pullback metric of the Euclidean metric $e$ along the inclusion $ i:\GLp \to (\mathbb{R}^{n^2},e) $. Explicitly: $g_z(X,Y) =tr(X^TY)$ , where $X,Y \in \operatorname{Mat}(n \times n, \R) \cong T_z(\GLp)$.
We denote by $\dint$ the distance function induced by the Riemannian metric $i^*g$.
\end{definition}
%%%%%%%%%

We first observe that $\dint > \dext$ for some pair of matrices. Indeed, for any $X,Y\in M_n$, the unique minimizing geodesic is the segment 
\[
[X,Y] = \{X+t\,(Y-X) ~:~ t\in[0,1]\}.
\]
Since the sub-manifold $\GLp$ is not convex, there exist $A,B\in\GLp$, such that the segment $[A,B]$ intersects $\GLm$. By Item~4 in \propref{prop:riemannian_intrinsic},
\[
\dext(A,B) < \dint(A,B).
\]

However, we note the following:

\begin{lemma}
\label{lem:equality_int_ext_convexity}
Let $A,B \in \GLp$. If $[A,B] \subset \GLp$, then $\dint(A,B)=\dext(A,B)$.
\end{lemma}

\begin{proof}
This is obvious, since $[A,B]$ is an extrinsic length-minimizing path which stays in the submanifold. 
\end{proof}

We next observe that both right- and left-multiplications by elements of $\SOn$ are isometries of $(GL_n^+,i^*g)$. Hence, they are isometries of the metric space $(GL_n^+,\dint)$ (any Riemannian isometry is an isometry of the induced distance function). It follows that $\dint$ is both left- and right-$\SOn$ invariant.

In particular, let $A = U\Sigma V^T$ be an SVD of $A\in\GLp$.
By \propref{prop:orthogonal_invariance_dist},
\[
\dint(A,\SOn) = \dint(\Sigma,\SOn),
\] 
hence, as before, the problem of computing the distance of $A\in\GLp$ from $\SOn$ (and finding the minimizer) 
can be reduced to positive-definite diagonal matrices $\Sig$.
We now give a short proof that $I$ is the unique matrix closest to $\Sig$ with respect to the extrinsic distance, that is,
\[
\dext(\Sig,\SOn) = \dext(\Sig,I) = \|\Sig-I \|_F.
\]
Indeed, since  
\beq
\|A - Q\|_F^2  = \|{A}\|_F ^2 + \|{Q}\|_F^2 - 2\IPF{A}{Q} = \|{A}\|_F^2 + n - 2\tr(A^TQ),
\eeq
it follows that given $A \in M_n$, minimizing $\|A-Q\|_F^2$ over $Q \in \SOn$ is equivalent to maximizing the linear functional  $\varphi_A(Q)=\IPF{A}{Q}=\tr(A^TQ)$.
For diagonal and positive-definite $\Sig$,

\[ 
\varphi_{\Sig}(Q) = \mathrm{tr}(\Sig^TQ) = \sum_{i = 1}^n \sig_i Q_{ii} \leq \sum_{i=1}^n \sig_i = \mathrm{tr}(\Sig^T) = \varphi_{\Sig}(I),
\]
where the inequality follows from the fact that $Q$ is orthonormal, hence $|Q_{ij}|\le1$.
The unique maximizer is $Q = I$, hence $I$ is the unique matrix closest to $\Sig$.

%Since $[\Sig,I] \subseteq GL_n^+$, \Raz{it follows that
%\[
%\dint(\Sig,\SOn) = \dext(\Sig,\SOn) = \|\Sig-I \|_F.
%\]}
%
%we conclude that the distance of $A\in\GLp$ to $\SOn$ with respect to $\dint$, as well as the minimizing element in $\SOn$, are the same:

%\begin{lemma}
 % \label{lem:orthogonal_distance_diagonal}
 % Let $\Sig$ be a diagonal matrix with non-negative entries on the diagonal. Then, the orthogonal matrix closest to $A$ w.r.t $\dext$ (Frobenius norm) is $I$. That is $I$ is the minimizer of $\min_{X \in\On} \dext(\Sig,X)=\min_{X \in\On} \Normtwo{\Sig-X}$. If all the entries of $\Sig$ are strictly positive, then $I$ is the unique minimizer.
%\end{lemma}

%%%%%%%%%%%
\begin{theorem}
\label{thm:coincidence_of_minimizer_ext_int}
Let $\Sig \in\GLp$ be a diagonal matrix with strictly positive entries on the diagonal. Then the unique minimizer of the intrinsic distance of $\Sig$ from $\SOn$ is $I$, and 
\[
\dint(\Sig,\SOn) = \dext(\Sig,\SOn).
\]
\end{theorem}
%%%%%%%%%%

%%%%%%%%%%
\begin{proof}
Since $\psym$ is closed under convex combinations, it follows that $[\Sig,I] \subseteq GL_n^+$.
By \lemref{lem:equality_int_ext_convexity}, it follows that $\dint(\Sig,I)=\dext(\Sig,I)$.

The inequality $\dint \ge \dext$ implies that
\[
\dist^{int}(\Sig,\SOn) \ge \dist^{ext}(\Sig,\SOn).
\] 
Hence,
\[ 
\dist^{ext}(\Sig,\SOn) = \dext(\Sig,I)=\dint(\Sig,I) \ge \dist^{int}(\Sig,\SOn) \ge \dist^{ext}(\Sig,\SOn).
\]
So, $I$ is a minimizer and $\dist^{int}(\Sig,\SOn) = \dist^{ext}(\Sig,\SOn)$.
  
The uniqueness of the minimizer for $\dist^{int}$ follows from the uniqueness of the minimizer for $\dist^{ext}$:
Let $Q \in \SOn$ be a minimizer for the distance $\dist^{int}$ of $\Sig$ from $\SOn$.
Then, 
\[
\dist^{ext}(\Sig,\SOn) =  \dist^{int}(\Sig,\SOn)=\dint(\Sig,Q) \ge  \dext(\Sig,Q),
\]
hence $Q$ is also a minimizer for  $\dist^{ext}$, and by the uniqueness of the extrinsic minimizer, $Q = I$.
\end{proof}
%%%%%%%%%%

\begin{corollary}
Let $A \in \GLp$. Then 
\[ 
\dint(A,\SOn) = \dext(A,\SOn), 
\] 
and the closest matrix to $A$ in $\SOn$ with respect to both distances is the same---it is the orthogonal polar factor of $A$. Moreover,
\beq
\dint(A,\SOn) = \|\sqrt{A^TA} - I\|_F.
\label{eq:grioli}
\eeq

\end{corollary}

\begin{proof}
This is an immediate consequence of
\propref{prop:orthogonal_invariance_dist} applied to both distances, together with \thmref{thm:coincidence_of_minimizer_ext_int}.
The fact that the orthogonal polar factor is the unique minimizer follows from the same considerations as in the proof of \corref{cor:Expression_for_the_Frobenius_minimizer}.
\end{proof}

%Equation~\ref{eq:grioli} is referred to in \cite{NEM15} as Grioli's optimality theorem. Please note how it is obtained here by very elementary means.

%%%%%%%%%%%%%%%%%%%%%%%%%%%%%%%%%%%%%
\section{Additional results concerning invariant distances}
\label{sec:invariant_distances}
%%%%%%%%%%%%%%%%%%%%%%%%%%%%%%%%%
\subsection{Bi-invariance}
\label{subsec:bi_invariance}
As always in mathematics, the most symmetric structures are the most easy to handle. 
In \cite{NEM15}, the authors consider either left- and right-$\GL$-invariance as natural requirements on a metric on $\GL$. Their choice is motivated by physical considerations.

A natural question is: does there exist a distance function on $\GL$ that is more symmetric than the ones we have considered? In this section we show here that there are no bi-invariant distance functions on $\GL$ that are compatible with the standard topology. (The fact there is no bi-invariant Riemannian metric is common knowledge.)

%In other words, these two requirements are incompatible with each other.

%%%%%%%%%%%
\begin{theorem}
\label{thm:no_bi_invariant metric}
There is no bi-invariant distance function on $GL_n$ generating the standard topology on $GL_n$ (the subspace topology induced by the inclusion $GL_n \to \R^{n^2}$).
\end{theorem}
%%%%%%%%%%

%%%%%%%%%%
\begin{proof}
The essential point is the existence of a non-trivial conjugacy class whose closure contains the identity.
Assume, by contradiction, there is a bi-invariant distance function $d$ compatible with the standard topology. 
Consider the following matrices:

%D=\begin{pmatrix}
%\lam_1 & 0 &  \cdots & 0 \\
% 0 & \lam_2 & \cdots & 0 \\
%\vdots  & \vdots& \ddots & \vdots \\
%0 & 0 & \cdots & \lam_n    
%\end{pmatrix}
\[
D=\begin{pmatrix}
2^{-1} & 0 &  \cdots & 0 \\
 0 & 2^{-2} & \cdots & 0 \\
\vdots  & \vdots& \ddots & \vdots \\
0 & 0 & \cdots & 2^{-n}    
\end{pmatrix}
\Textand
A=\begin{pmatrix}
1 & 1 & 0 & \cdots & 0 \\
0 & 1 & 1 &  \cdots & 0 \\
0 & 0 & 1& \ddots & 0 \\
 \vdots  & \vdots& \vdots & \ddots & \vdots & \\  
0 & 0 & 0& \cdots  & 1 
\end{pmatrix}.
\]
An explicit calculation yields,
%Since left (right) multiplication by a diagonal matrix amounts to multiplying each row (column) by the corresponding diagonal element we obtain:
%  $DA=\begin{pmatrix}
%    \lam_1 & \lam_1a_1 & 0 & \cdots & 0 \\
%    0 & \lam_2 & \lam_2a_2& \cdots & 0 \\
%    0 & 0 & \lam_3& \ddots & 0 \\
%    \vdots  & \vdots& \vdots & \ddots & \vdots & \\  
%    0 & 0 & 0& \cdots  & \lam_n 
%  \end{pmatrix} \Rightarrow 
\[
D^{-n}AD^{n}=\begin{pmatrix}
1 & 2^{-n} & 0 & \cdots & 0 \\
0 & 1 & 2^{-n} & \cdots & 0 \\
0 & 0 & 1& \ddots & 0 \\
\vdots  & \vdots& \vdots & \ddots & \vdots & \\  
0 & 0 & 0& \cdots  & 1 
\end{pmatrix}.
\]

%  So $DAD^{-1}$ has $1$'s on the diagonal, and above it: $(DAD^{-1})_{i,{i+1}}=\frac{\lam_i}{\lam_{i+1}}a_i:=b_i$
%  Hence $D^2AD^{-2}$ has also $1$'s on the diagonal and $(D^2AD^{-2})_{i,{i+1}}=\frac{\lam_i}{\lam_{i+1}}b_i=(\frac{\lam_i}{\lam_{i+1}})^2a_i$ above it. Similarly $(D^nAD^{-n})_{i,{i+1}}=(\frac{\lam_i}{\lam_{i+1}})^na_i$
%
%  In other words, conjugation leaves the diagonal invariant and multiplies each entry above the diagonal by the ratio of the corresponding diagonal elements of the conjugate matrix.
%
%
%  If we choose for instance all $a_i=1$ and $\lam_i$ such that $\frac{\lam_i}{\lam_{i+1}}=\frac{1}{2}$, then $D^nAD^{-n}$ converges entrywise to the identity $I$, hence converges w.r.t the subspace topology on $GL_n$ induced by $\R^{n^2}$.

Since $d$ generates the standard topology on $GL_n$, 
it follows that 
\[
  \lim_{n\to\infty} d(D^{-n}AD^{n},I) = 0.
\]
On the other hand, bi-invariance implies that for every $n$,
\[
  d(D^{-n} A D^{n},I) = d(D^{-n} A D^{n},D^{-n} I D^{n})  = d(A,I) \ne 0,
\]
which is a contradiction.
\end{proof}

As an immediate corollary we obtain the classical result:

\begin{corollary}
There is no bi-invariant Riemannian metric on $\GL$.
\end{corollary}

%%%%%%%%%%%%%%%%%%%%%%%%%%%%%%%
\subsection{Inverse-invariance}

When it comes to physical applications, a major drawback of the Euclidean strain measure is that it does not diverge in the limit where the linear map is singular.
From a physical viewpoint, we expect a strain measure $\distSO{A}$ to diverge when $A$ either tends to infinity (expansion), or when it tends toward singularity (contraction). The strain measure is said to be \Emph{inverse-invariant} if
\beq
%\distSO{A}=\distSO{A^{-1}} 
\strain{A}=\strain{A^{-1}}, 
\label{eq:contraction_expansion}
\eeq
for every $A \in \GLp$.

As noted in the Introduction, the strains obtained via the metrics considered in \secref{sec:symm_geo} are all inverse-invariant.
The essential reason behind this phenomenon, is the extreme symmetry of these metrics. Specifically, we have the following very general assertion:

\begin{proposition}
\label{prop:inverse_invariance_general_Lie_groups}
Let $G$ be a group and $H \subseteq G$ a subgroup. Let $d$ be a left-$G$- right-$H$-invariant distance function on $G$. Then,
\[
\dist(g,H) = \dist(g^{-1},H).
\]

Moreover, if $h$ is a closest element to $g$ in $H$, then $h^{-1}$ is a closest element to $g^{-1}$ in $H$.
\end{proposition}

\begin{proof}
Using the assumed invariances,
\[ 
\dist(g,H) = \inf_{h \in H} d(g,h) = \inf_{h \in H} d(e,g^{-1}h) = \inf_{h \in H} d(h^{-1},g^{-1})= \dist(g^{-1},H) 
\]

The equality $d(g,h)=d(g^{-1},h^{-1})$ (for $h \in H$) implies the correspondence between closest elements to $g$ and $g^{-1}$.
\end{proof}

Note how this observation implies, without any computation, that certain strain measures are inverse-invariant  (the actual form of the strain measures is irrelevant).

Another means for obtaining inverse-invariant strain measures is to consider distance functions $d$ on $\GL$ that are inverse-invariant, i.e.,
\beq
\label{eq:inverse_symmetric_distance}
d(A,B)=d(A^{-1},B^{-1}).
\eeq
Indeed, the inverse-invariance of the distance implies the inverse-invariance of the strain measure, as
\[
\distSO{A}=\dist(A^{-1},\SOn^{-1})=\distSO{A^{-1}}.
\]

As we will see, the requirement for inverse-invariance of the distance is far less restrictive than left-$\GL$- and right $\On$-invariance.

In a search for maximally-symmetric distance functions, we first look for inverse-invariant distances (or metrics) possessing additional symmetries. In Propositions~\ref{prop:inverse_invariant_metrics_Lie}--\ref{prop:inverse_invariant_distances_Lie}
we prove that in any Lie group $G$, a left-invariant metric (distance) is inverse-invariant if and only if it is bi-invariant.  Since there are no bi-invariant metrics/distances on $\GL$, it follows that no left-$\GL$- and inverse-invariant metric/distance exists either.  

In particular, it follows that the distance functions induced by the metrics considered in Sections~\ref{sec:symm_geo} and \ref{sec:geo_dist} are not inverse-invariant. %(This follows from the Myers Steenrod theorem, which says that any isometry of the distance induced by a metric, is an isometry of that metric).
That is, the inverse-invariance of the strain measure does not result from the inverse-invariance of the distance, but rather from the left-$\GL$ and right-$\SOn$ invariance of the metrics, as shown in \propref{prop:inverse_invariance_general_Lie_groups} (also observed in ~\cite[Section 3.2 Eq.~(23)]{NEM15}).

%what about the ``inversification'' of all the metrics we have considered so far?

%Actually, it seems (after a short calculation) that for the standard left-invariant Riemannian metric on $\GL$ , $g+i^*g$ is bi-$\On$-invariant, so we do have a natural Riemannian metric in that sense.

There is a systematic way of constructing inverse-invariant distances from arbitrary distances.
Denote the inverse automorphism by $i$. Since $i$ is a diffeomorphism of finite order, 
given any distance $d$ on $\GL$,
it is possible to construct an inverse-invariant distance via symmetrization,
\[
\til d(A,B)=d(A,B)+d(A^{-1},B^{-1}).
\]
It is easy to see that $\til d$ generates the same topology as $d$, and it is of course inverse-invariant.

A similar construction can be carried out for Riemannian metrics on $\GL$. Given any metric $g$, the metric $g + i^*(g)$ is inverse-invariant, and induces the standard topology on $\GL$, as does any Riemannian metric.

In the following subsections, we analyze how these two different methods for generating inverse-invariant distances affect the strain measure.
We will see that if we start from distances/metrics having certain symmetries, then their symmetrizations possess corresponding symmetries as well.

%%%%%%%%%%%%%%%%%%%%%%%%%%
\subsubsection{Inverse-invariant distances}

%We have so far discussed two sorts of distance functions. The first are those distances which were induced by the left $\GL$- and right-$\On$ invariant metrics, considered in section ~2. The second kind are the intrinsic and extrinsic Euclidean distance functions considered in section ~4.3. We treat the first kind first, using only there symmetries.

The next lemma considers a setting that generalizes our treatment of left-$\GL$- and right-$\On$ invariant distance functions.

%%%%%%%%%%%
\begin{lemma}
Let $G$ be a group and let $H \subseteq G$ be a subgroup. Let $d$ be a left-$G$- and right-$H$-invariant distance function on $G$. Let $\til d$ be its symmetrization. Then,
\[
\dist_{\til d}(g,H)=2 \,  \dist_d(g,H).
\]
Moreover, an element $h \in H$ is a  closest element in $H$ to $g$ with respect to $d$ if and only if it is a closest element with respect to $\til d$.
\end{lemma}
%%%%%%%%%%%

%%%%%%%%%%%
\begin{proof}
The assumed symmetries imply $d(g,h)=d(g^{-1},h^{-1})$ for any $g\in G$ and $h \in H$ (see the proof of \propref{prop:inverse_invariance_general_Lie_groups}).
Hence,
\[   
\begin{split}
\dist_{\til d}(g,H) &= \inf_{h \in H} \til d(g,h) \\
&= \inf_{h \in H} d(g,h) + d(g^{-1},h^{-1}) \\
&=  \inf_{h \in H} 2 \, d(g,h) \\
&= 2 \,  \dist_d(g,H).
\end{split}
\]
To prove the second part, note that for $h\in H$, $\til d(g,h) = \dist_{\til d}(g,H)$ if and only if 
$2\, d(g,h) =  2 \, \dist_d(g,H)$.

\end{proof}
%%%%%%%%%%

The above lemma implies that there is no much interest in symmetrizing left-$G$- and right-$H$-invariant distance functions, since the symmetrizations give rise to essentially identical notions of distance from $H$. 

%As stated above, this applies in particular for all the distance functions induced by the metrics considered in Section~2.

We turn to analyze symmetrization within the context of intrinsic versus extrinsic Euclidean distances. As above, we provide a slightly more general treatment that considers the relevant symmetries.
Since this setting possesses less symmetries than the one considered above,  we will have to use properties that are specific to the Euclidean distance; in particular, SVD plays an important role. 

%%%%%%%%%
\begin{lemma}
Let $G$ be a group and let $H \subseteq G$ be a subgroup. Let $d$ be  a bi-$H$-invariant distance on $G$. Then, its symmetrization $\til d$ is also bi-$H$-invariant.   
\end{lemma}
%%%%%%%%%

%%%%%%%%%
\begin{proof}
For every $h \in H, x,y \in G$, 
\[
\til d(hx,hy)=d(hx,hy)+d(x^{-1}h^{-1},y^{-1}h^{-1})=d(x,y)+d(x^{-1},y^{-1}) = \til d(x,y).
\]
Right-$H$-invariance is proved similarly.
\end{proof}
%%%%%%%%

In our context, since both intrinsic and extrinsic Euclidean distances are bi-$\SOn$ invariant, their symmetrizations are also bi-$\SOn$ invariant.
 
By \propref{prop:orthogonal_invariance_dist}, $\tildistSO{A}=\tildistSO{\Sig}$, where $A = U\Sig V^T$ is any SVD of $A$. By the results in Section~\ref{subsec:intrinsic_euc}, $I$ is the closest matrix to both $\Sig$ and $\Sig^{-1}$ with respect to both intrinsic and extrinsic Euclidean distances.

Hence, for every $Q \in \SOn$,
\[
\til d(\Sig,Q)=d(\Sig,Q)+d(\Sig^{-1},Q^{-1}) \ge d(\Sig,I)+d(\Sig^{-1},I)=\til d (\Sig,I),
\]
from which follows that $I$ is the matrix in $\SOn$ that is the closest to $\Sig$ with respect to $\til d$, and
\[ 
\begin{split}
\tildistSO{A} & =\tildistSO{\Sig}= \til d(\Sig,I) = d(\Sig,I)+d(\Sig^{-1},I) \\
&= \dist_d(A,\SOn)+\dist_d(A^{-1},\SOn).
\end{split}
\]
Again, we obtain that the matrix in $\SOn$ that is the closest to $A$ is the orthogonal polar factor of $A$.
 
%Then $\tildistSO{A}=2 \,  \origindistSO{A}$, i.e the distance from $\SOn$ is doubled. Moreover the minimizer stays the same (its the orthogonal polar factor).
  
The symmetrization of the Euclidean distance gives a truly different notion of strain measure, as it penalizes equally both expansions and contractions. Thus, it can be considered an improved strain measure. At the same time, it preserves the symmetries pertinent to the Euclidean metric---frame invariance and material isotropy.
%
%. From a physical perspective, it is an improved notion, since now there is energy penalty on  Since the left-invariant distances considered in the section before already gave rise to inverse-invariant distances from $\SOn$ it is not surprising no real change was made through their symmetrization.

 % \begin{lemma}
 %   Let $d$ be left- and right-$\SOn$-invariant distance on $\GLp$. Let $\til d$ be its symmetrization. Then $\tildistSO{A}=2 \,  \origindistSO{A}$, i.e the distance from $\SOn$ is doubled. Moreover the minimizer stays the same (its the orthogonal polar factor).
 % \end{lemma}

 % \begin{proof}
 %   First we show $\til d$ inherits the left- and right-$\SOn$-invariance from $d$. For every $Q \in \SOn$
 %   \[\til d(QA,QB)=d(QA,QB)+d(A^{-1}Q^{-1},B^{-1}Q^{-1})=d(A,B)+d(A^{-1},B^{-1}) = \til d(A,B)\]
 %   Right-$\SOn$-invariance is proved similarly. 
 % \end{proof}

%%%%%%%%%%%%%%%%%%%%%%%%
\subsubsection{Inverse-invariant metrics}

We now turn to the symmetrization of Riemannian metrics on $\GL$. Given any metric $g$, the metric $\til g =g + i^*(g)$ is inverse-invariant.
As proved in \cite{1623388}, the space of inverse-invariant metrics for any Lie group is infinite-dimensional, hence, we have to address the question of finding natural inverse-invariant metrics. We note that in general, the distance function induced by the symmetrized metric $\til g$ is \Emph{not} the symmetrization of the distance function induced by $g$. Hence, the analysis of the previous subsection is not applicable.

We start by showing that the symmetrized metric inherits some of the symmetries of the original metric.

%%%%%%%%%
\begin{lemma}
\label{lem:symmetrized_metric_inherits_symmetries}
Let $G$ be a Lie group and let $H \subseteq G$ be a subgroup. Let $g$ be a bi-$H$-invariant metric on $G$. Its symmetrization $\til g$ is also bi-$H$-invariant. 
\end{lemma}
%%%%%%%%%

%%%%%%%%%%
\begin{proof}
Let $h \in H$. Then 
\[ 
L_h^*(g+i^*g)=L_h^*g + L_h^*(i^*g)=g+(i \circ L_h)^*g.
\]
Since $i \circ L_h = R_{h^{-1}} \circ i$,    
\[  L_h^*(g+i^*g)= g +  (R_{h^{-1}} \circ i)^*g=g + i^*(R_{h^{-1}}^*g)=g+i^*g .
\]
The proof that $R_h^*(g+i^*g)=g+i^*g $ is similar.

%
%  Let $O \in \On$. Then \[ L_O^*(g+i^*g)=L_O^*g + L_O^*(i^*g)=g+(i \circ L_O)^*g, \]
%
%  Since $i \circ L_O = R_{O^{-1}} \circ i$ we get:   \[  L_O^*(g+i^*g)= g +  (R_{O^{-1}} \circ i)^*g=g + i^*(R_{O^{-1}}^*g)=g+i^*g \]
%
%  The proof that $R_O^*(g+i^*g)=g+i^*g $ is similar.

\end{proof}
%%%%%%%%

It follows that the symmetrizations of all the metrics considered in sections~\ref{sec:symm_geo} and \ref{subsec:intrinsic_euc}  are bi-$\On$ invariant.

In the remaining part of this section we study the symmetrization of the metrics considered in Section~\ref{sec:symm_geo}.

%Thus we have again two kinds of bi-$\On$- and inverse-invariant metrics, which seem to be natural candidates for investigation; The first kind is the family of all symmetrizations of the metrics considered in section ~2. The other one is the symmetrization of the intrinsic Euclidean metric. We start by analyzing the first kind.
As mentioned above, the symmetrized metrics, unlike the original metrics, are not left-$\GL$ invariant, hence, our analysis of the geodesics is not applicable. 
However, the symmetrized metrics share  three important properties with the original metrics:
\begin{enumerate}
\item $(\GLp,\til g)$ is complete. 
\item The symmetric and the skew-symmetric matrices are orthogonal with respect to $\til g_I$.
\item $\al(t)=e^{tV}$ is a $\til g$-geodesic for any symmetric matrix $V$.
\end{enumerate}

We start by showing the orthogonality of symmetric and skew-symmetric matrices. 
For any $A \in \symn$ and $B \in \antisym$, 
\[
(i^*g)_I(A,B)=g_I(di_I(A),di_I(B))=g_I(-A,-B)=0 ,
\]
hence
\[
\til g_I(A,B)=g_I(A,B)+(i^*g)_I(A,B) = 0.
\]

We proceed to prove the completeness of $(\GLp,\til g)$. First note that completeness of $g$ implies completeness of $i^*g$, since $i:(\GL,i^*g) \to (\GL,g)$ is an isometry. So, it suffices to prove that if two Riemannian manifolds $(M,g_1)$ and $(M,g_2)$ are complete then so is $(M,g_1+g_2)$.

%%%%%%%%%%%%%
\begin{lemma}
Let $M$ be a smooth manifold, and let $g_1,g_2$ be Riemannian metrics on $M$. If either $(M,g_1)$ or $(M,g_2)$ is complete, then $(M,g_1+g_2)$ is complete.
\end{lemma}
%%%%%%%%%%

%%%%%%%%%
\begin{proof}
Let $p,q \in M$. For any path $\al$ from $p$ to $q$,
\[
\begin{split}
L_{g_1+g_2}(\al) &= \int \sqrt{g_1(\dot \al(t),\dot \al(t))+ g_2(\dot \al(t),\dot \al(t))} \\
&\ge \int \sqrt{g_1(\dot \al(t),\dot \al(t))} \\
&= L_{g_1}(\al).
\end{split} 
\]
Similarly $L_{g_1+g_2}(\al) \ge L_{g_2}(\al)$. Without loss of generality,  Assume that $(M,g_1)$
 is complete.
For any $p,q \in M$,
\[ 
d^{g_1}(p,q)=\inf_{\al :p\mapsto q}L_{g_1}(\al) \le  \inf_{\al :p\mapsto q}L_{g_1+g_2}(\al) =d^{g_1+g_2}(p,q)
\]
By the Hopf-Rinow theorem, a Riemannian manifold $(M,g)$ is complete if and only if closed and $g$-bounded sets are compact. Let $A \subseteq M$ be a closed and $(g_1+g_2)$-bounded set. Boundedness implies that there exists a point $p \in M$ and a number $R > 0$, such tha $d^{g_1+g_2}(a,p) \le R$ for every $a\in A$.

Since $d^{g_i}(a,p) \le d^{g_1+g_2}(a,p) \le R$, it follows that $A$ is also $g_1$-bounded. Since $(M,g_1)$ is complete, $A$ is compact, hence $g_1+g_2$ is complete as well.
\end{proof}
%%%%%%%%

%  Note: The above proof shows that any metric which is (pointwise) bigger than a complete one is complete. Myabe we should write this as a lemma, and then deduce what we need as a corollary.

It remains to show that $\al(t)=e^{tV}$ is a geodesic for symmetric $V$.
We first note that $\al(t)=i(e^{tV})=e^{-tV}$ is a geodesic of $i^*g$, since $i:(\GL,i^*g) \to (\GL,g)$ is an isometry. Reversing time, we get that $\al(t)=e^{tV}$ is a geodesic of both $g$ and $i^*g$. 
Note that every geodesic is parametrized by a parameter proportional to arclength, i.e., its speed $\|\dot \al(t)\|$ is constant. In this particular case, the speeds are the same when measured with respect to both metrics, since $g_I=(i^*g)_I$, hence $\|\dot \al(0) \|$ is independent of the metric chosen. It turns out that in this particular situation, $\al$ is a geodesic with respect to the metric $g+i^*g$.

%; Since it is a geodesic, it has constant speed w.r.t $g,i^*g$. The speeds are identical. (Since they are constant we can check this at a single point,$I$, where $g_I(W,W)=(i^*g)_I(W,W)$).

%We now prove the following:

%%%%%%%%
\begin{lemma}
  \label{lem:geodesic_metric_sum_lemma}
  Let $M$ be a smooth manifold, and let $g_1,g_2$ be Riemannian metrics on $M$. Assume $\be(t)$ is a geodesic for both $g_1$ and $g_2$, and that its speed is the same with respect to both metrics.  Then $\beta(t)$ is also a $(g_1+g_2)$-geodesic.
\end{lemma}
%%%%%%%%
  
%%%%%%%%
\begin{proof}
The inequality: $\sqrt{a+b} \ge \frac{1}{\sqrt 2}(\sqrt a + \sqrt b) $ implies that for any path $\al$ in $M$,
\beq 
\label{eq:length_inequality}
\begin{split} 
L^{g_1+g_2}(\al) &= \int \sqrt{g_1(\dot \al(t),\dot \al(t))+ g_2(\dot \al(t),\dot \al(t))} \\
& \ge  \frac{1}{\sqrt 2} \cdot \big( \int  \sqrt{g_1(\dot \al(t),\dot \al(t))} +   \int  \sqrt{g_2(\dot \al(t),\dot \al(t))} \big) \\ 
&=  \frac{1}{\sqrt 2} \cdot (L^{g_1}(\al)+L^{g_2}(\al)).
\end{split} 
\eeq
Since $\be$ is a geodesic with respect to both $g_1$ and $g_2$, it is locally length-minimizing with respect to both metrics; for small enough $t$, $L^{g_i}(\be|_{[ 0,t]})=d^{g_i}(\be(0),\be(t))=tc$, for some constant $c$. Let $\al$ be any path connecting $\be(0),\be(t)$. By our assumption, $\sqrt{g_1(\dot \be(t),\dot \be(t))}= \sqrt{g_2(\dot \be(t),\dot \be(t))}=c$.
Note that 
\[ 
\begin{split} 
  L^{g_1+g_2}(\be|_{[0,t]}) &= \int_0^t \sqrt{g_1(\dot \be(t),\dot \be(t))+ g_2(\dot \be(t),\dot \be(t))} \\
%&=   \int_0^t \sqrt{2} \\
  &= \sqrt{2} ct= \frac{1}{\sqrt{2}}\cdot(tc+tc)  \\
&=  \frac{1}{\sqrt 2}\cdot \big(d^{g_1}(\be(0),\be(t))+ d^{g_2}(\be(0),\be(t))\big) \\ 
&\le   \frac{1}{\sqrt 2} \cdot (L^{g_1}(\al)+L^{g_2}(\al))\le L^{g_1+g_2}(\al)
\end{split}, 
\]
where the last inequality uses \eqref{eq:length_inequality}. Thus, 
$\be$ locally minimizes length with respect to $g_1+g_2$, hence it is a geodesic.
\end{proof}
%%%%%%%%

Next, we imitate from Section~\ref{sec:geo_dist} the argument for finding the geodesic distance from $\SOn$. By \lemref{lem:symmetrized_metric_inherits_symmetries}, we can use \propref{prop:orthogonal_invariance_dist} to reduce again the question to diagonal matrices.
Since the derivation of the strain measure uses only the three properties of the metric mentioned above, it works in exactly the same manner for the symmetrized metric.

There is just one delicacy. The proof hinges on the fact that $\al(t)=Qe^{tV}$ is a $\til g$-geodesic for any symmetric $V$ and $Q \in \SOn$ (note that at this stage we don't yet know that $Q=I$).
From \lemref{lem:geodesic_metric_sum_lemma} follows that $e^{tV}$ is a $\til g$-geodesic. Since (\lemref{lem:symmetrized_metric_inherits_symmetries} again) $\til g$ is bi-$\On$-invariant it follows that $\al$ is also a $\til g$ geodesic.
Following the rest of the proof, the only difference is at the final stage, when evaluating the speed $\|\dot \al(0)\|_I$, where $\til g_I$-is scaled by $\sqrt{2}$,
\[
\|V\|_I=\sqrt{g_I(V,V)+(i^*g)_I(V,V)}=\sqrt{2\cdot g_I(V,V)}.
\]
Hence, the strain measure is multiplied by a factor of $\sqrt{2}$.

%%%%%%%%%%%%%%%%%%%%%%%%%%%%%%%%%%%%%
\appendix
%\section{Uniqueness Property of Singular Values Decomposition}
%
%\begin{lemma}[Uniqueness of $UV^T$ in SVD]
%  \label{lem:uniqueness_in_SVD}
%  Let $A = UDV^T$ be an SVD decomposition of an invertible matrix. Then the product $UV^T$ is independent of the particular decomposition.
%\end{lemma}
%
%\begin{proof}
%  Assume $A=UDV^T=U'DV'^T$ are two SVD decompositions of $A$. Then we can write $U'=UO$, where $O \in \On$. We claim $O,D$ commute.
%  
%  $AA^T=UD^2U^T=U'D^2U'^T=(UO)D^2(UO)^T=U(OD^2O^T)U^T \Rightarrow  D^2=OD^2O^T=(ODO^T)^2$
%
%  Hence, $D,ODO^T$ are both symmetric positive-definite matrices, whose square equals $D^2$, hence by the uniquness of the square root this forces $D=ODO^T=ODO^{-1} \Rightarrow DO=OD$.
%
%  Hence we obtain \[V'^T=(U'D)^{-1}A=D^{-1}U'^{-1}A=D^{-1}O^{-1}U^{-1}A=O^{-1}D^{-1}U^{-1}A=O^{-1}(UD)^{-1}A=O^{-1}V^T \]
%
%  So, $U'V'^T=(UO)(O^{-1}V^T)=UV^T$.
%\end{proof}
%
%\begin{lemma}
%  \label{SVD_decomp_SO}
%  Let $A \in \GLp$. Then there exists an SVD-decomposition of $A=U\Sig V^T$ where $U,V \in \SOn$.
%\end{lemma}
%\begin{proof}
%  The argument is essentially contained in the proof of \corref{cor:Expression_for_the_Frobenius_minimizer}.
%  Let $A=OP$ be the polar decomposition of $A$ ($O \in \SOn, P \in \psym$).
%
%  By orthogonally diagonalizing $P$ with $P = \til U \Sigma \til U^T$, we obtain an $SVD$ decomposition $A = O\til U\Sigma \til U^T=U \Sigma V^T$ for $A$, where $U=O\til U,V=\til U$. Note that by switching two columns if necessary, we can assume w.l.o.g that $\til U \in \SOn$, hence $V,U \in \SOn$.
%\end{proof}

\section{Calculating the geodesics}
\label{sec:geosesic_eqs}

\subsection{Analysis of the geodesic equations}

Let $g_I$ be an inner-product on $M_n=T_I\GL$ given by the form \eqref{eq:expression_inner_product}, and let $g$ be the Riemannian metric on $\GL$ which is the left-translation of $g_I$.

First, we need the following result.

\begin{proposition}
\label{prop:G1}
For every $X,Y\in T_I\GL$,
\[
\GG(X,[Y,X]) = \frac{\beta - \gamma}{2\beta} \GG([X,X^T],Y).
\]
\end{proposition}
%%%%%%%%%

%%%%%%%
\begin{proof}
Note first that
\[
  \GG(X,[Y,X]) = \beta\tr(\sym X\,\sym [Y,X]) + \gamma \tr(\skew X\,\skew [Y,X]),
\]
and
\[
\GG([X,X^T],Y) = \beta\tr(\sym [X,X^T]\,\sym Y).
\]

Now,
\[
\begin{split}
\tr(\sym X\,\sym [Y,X]) &= \frac14 \tr((X+X^T)(YX-XY + X^TY^T - Y^TX^T)) \\
%&\hspace{-3cm}= \frac14 \tr(XYX-XXY + XX^TY^T - XY^TX^T + X^TYX- X^TXY + X^TX^TY^T - X^TY^TX^T) \\
%&\hspace{-3cm}= \frac14 \tr(XX^TY^T - X^TXY^T + XX^TY- X^TXY) \\
&\hspace{-3cm}= \frac12\tr([X,X^T]Y) 
= \frac{1}{2\beta} \GG([X,X^T],Y),
\end{split}
\]
and
\[
\begin{split}
\tr(\skew X\,\skew [Y,X]) &= \frac14 \tr((X-X^T)(YX-XY - X^TY^T + Y^TX^T)) \\
%&\hspace{-3cm}= \frac14 \tr(XYX-XXY - XX^TY^T + XY^TX^T - X^TYX+ X^TXY + X^TX^TY^T - X^TY^TX^T) \\
%&\hspace{-3cm}= \frac14 \tr( - XX^TY^T + XY^TX^T - X^TYX+ X^TXY) \\
%&\hspace{-3cm}= \frac14 \tr( - XX^TY + X^TXY - XX^TY+ X^TXY) \\
&\hspace{-3cm}= -\frac12 \tr( [X,X^T]Y) 
= -\frac{1}{2\beta} \GG([X,X^T],Y),
\end{split}
\]
hence
\[
\GG(X,[Y,X]) = \frac{\beta - \gamma}{2\beta} \GG([X,X^T],Y).
\]
\end{proof}
%%%%%%

%%%%%%%%%%%
%\begin{proposition}
%\begin{shaded}
%$\{e_\alpha\} = \{\eta(\fraka_\alpha)\}$ is an orthonormal frame for $\G$.
%\end{shaded}
%\end{proposition}
%%%%%%%%%%
%
%%%%%%%%
%\begin{proof}
%This is an immediate consequence of Corollary~\ref{cor:4.1}.
%\end{proof}
%%%%%%%

Let $\{\fraka_\alpha\}$ be a $\GG$-orthonormal basis for $(T_I\GL,g)$.
Since $g$ is left-invariant, $e_\al=d(L_g)_e(\fraka_\al)$  is an orthonormal frame for $(T\GL,g)$.
Let $\{\cof\alpha\}$ be the orthonormal co-frame,
\[
\cof\alpha|_A = \G_A(e_\al|_A,\cdot) = \G_A(i_A(A\fraka_\alpha),\cdot),
\]
which implies that
\[
\cof\alpha|_A(i_A(AX)) = \GG(\fraka_\alpha,X).
\]

The Riemannian connection is represented by an anti-symmetric matrix of 1-forms, $\{\conn\alpha\beta\}$, defined by
\[
\nabla_{e_\alpha} e_\beta = \conn\gamma\beta(e_\alpha)\,e_\gamma, 
\]
and satisfying \Emph{Cartan's first structural equation},
\[
d\cof\alpha + \conn\alpha\beta\wedge\cof\beta = 0.
\]
Noting that,
\[
\begin{split}
d\cof\alpha(e_\mu,e_\nu) &=  (\cof\alpha(e_\nu))\,e_\mu -(\cof\alpha(e_\mu))\, e_\nu - \cof\alpha([e_\mu,e_\nu]) \\
&= \del^\al_\nu\, e_\mu  - \del^\al_\mu\,e_\nu - \cof\alpha([e_\mu,e_\nu])  - \G(e_\alpha,[e_\mu,e_\nu]) ,
\end{split}
\]
and that
\beq
\G(e_\alpha,[e_\mu,e_\nu])  = 
\GG(\fraka_\alpha,[\fraka_\mu,\fraka_\nu]),
\label{eq:explain_why}
\eeq
we get
\[
-\GG(\fraka_\alpha,[\fraka_\mu, \fraka_\nu])
 + \conn\alpha\nu(e_\mu) - \conn\alpha\mu(e_\nu) = 0.
\]
Equality \eqref{eq:explain_why} holds because Lie brackets of left-invariant vector fields are left-invariant, together with the well-known fact that the Lie algebra commutator of $\GL$ is merely the standard matrix commutator in $M_n$, i.e $[e_\mu,e_\nu](e)=[\fraka_\mu,\fraka_\nu]$ \cite[p. 193]{Lee12}.

Rotating the indexes,
\[
\begin{gathered}
-\GG(\fraka_\mu,[\fraka_\alpha, \fraka_\nu])
 + \conn\mu\nu(e_\alpha) - \conn\mu\alpha(e_\nu) = 0 \\
-\GG(\fraka_\nu,[\fraka_\mu, \fraka_\alpha])
 + \conn\nu\alpha(e_\mu) - \conn\nu\mu(e_\alpha) = 0.
\end{gathered}
\]
Adding the three equations and renaming the indexes,
\[
2\conn\gamma\beta(e_\alpha) =
\GG(\fraka_\alpha,[\fraka_\gamma, \fraka_\beta])
+\GG(\fraka_\gamma,[\fraka_\alpha, \fraka_\beta])
+\GG(\fraka_\beta,[\fraka_\gamma, \fraka_\alpha]).
\]
That is,
\[
\nabla_{e_\alpha} e_\beta = \Gamma^\gamma_{\alpha\beta}\,e_\gamma, 
\]
where 
\[
\Gamma^\gamma_{\alpha\beta} =  \half\brk{
\GG(\fraka_\alpha,[\fraka_\gamma, \fraka_\beta])
+\GG(\fraka_\gamma,[\fraka_\alpha, \fraka_\beta])
+\GG(\fraka_\beta,[\fraka_\gamma, \fraka_\alpha])}
\]
are constant coefficients. 

Consider now a geodesic curve, $\gamma:I\to\GL$, where
\[
\dot{\gamma}(t) = p^\alpha(t)\, e_\alpha|_{\gamma(t)}.
\]
The geodesic equation for the coefficients $p^\alpha(t)$ is
\[
\dot{p}^\gamma(t) + \Gamma^\gamma_{\alpha\beta} \, p^\alpha(t)\, p^\beta(t) = 0.
\]
Exploiting the symmetries of $\Gamma$ and the symmetry of the geodesic equation,
\[
\dot{p}^\gamma(t) + \GG(\fraka_\alpha,[\fraka_\gamma, \fraka_\beta])  \, p^\alpha(t)\, p^\beta(t) = 0.
\]
This is a set of $n^2$ quadratic equations with constant coefficients. 

Multiplying this equation by $\fraka_\gamma$, setting $X(t) = p^\alpha(t)\fraka_\alpha$, which is a curve in $T_I\GL$,
\[
\dot{X}(t) + \GG(X(t),[\fraka_\gamma, X(t)])  \fraka_\gamma = 0,
\]
which by  \propref{prop:G1},
\[
\begin{split}
0 &= \dot{X} + \GG(X,[\fraka_\gamma, X])  \fraka_\gamma 
= \dot{X} + \frac{\beta - \gamma}{2\beta}\GG([X,X^T],\fraka_\gamma)  \fraka_\gamma,
\end{split}
\]
namely,
\beq
\dot{X} = \kappa (X^TX - XX^T),
\label{eq:geodX}
\eeq
where
\[
  \kappa = \frac{\beta - \gamma}{2\beta}.
\]
Equation \eqref{eq:geodX} is an ordinary differential system in the vector space $M_n$.

The geodesic $\gamma:I\to\GL$ is related to $X(t)$ via,
\beq
\dot{\gamma}(t) = i_{\ga(t)}(\gamma(t),\gamma(t) X(t)).
\label{eq:geodGamma}
\eeq

\subsection{Solution of geodesic equations}

The factor $(\beta - \gamma)/2\beta$ has for effect to rescale time.
We start by ignoring it.

%%%%%%%%%%
\begin{proposition}
%\begin{shaded}
The solution to 
\[
\dot{X} = X^T X - X X^T
\qquad
X(0) = X_0,
\]
where $X:I\to M_n$
is
\[
X(t) = \exp(t(X_0^T - X_0)) X_0 \exp(t(X_0 - X_0^T)).
\]
%\end{shaded}
\end{proposition}
%%%%%%%%%

%%%%%%%
\begin{proof}
Clearly, the initial conditions are satisfied. Differentiating with respect to $t$ we get
\[
\begin{split}
\dot{X}(t) &= \exp(t(X_0^T - X_0)) (X_0^T - X_0)X_0 \exp(t(X_0 - X_0^T)) \\
&\,\,\,+
\exp(t(X_0^T - X_0)) X_0(X_0 - X_0^T) \exp(t(X_0 - X_0^T)) \\
&= \exp(t(X_0^T - X_0)) (X_0^TX_0 - X_0X_0^T) \exp(t(X_0 - X_0^T)).
\end{split}
\]
It only remains to insert $\exp(t(X_0 - X_0^T)) \exp(t(X_0^T - X_0))$ inside the products in the middle term to get
\[
\dot{X}(t) = X^T(t) X(t) - X(t) X^T(t).
\]
\end{proof}
%%%%%%

Please note that this exponential is the ``standard" matrix exponential, i.e.,  the one obtained from integral curves of left-invariant vector fields. It is not the exponential map of the $\G$-geodesics.

%%%%%%%%%
\begin{corollary}
%\begin{shaded}
The solution to \eqref{eq:geodX} is
\[
X(t) = \exp(\kappa t(X_0^T - X_0)) X_0 \exp(\kappa t(X_0 - X_0^T)).
\]
%\end{shaded}
\end{corollary}
%%%%%%%%

%%%%%%%%%%
\begin{proposition}
%\begin{shaded}
Let $\gamma:I\to\GL$ be the $\G$-geodesic,
\[
  \dot{\gamma}(t) = i_{\ga(t)}(\gamma(t),\gamma(t) X(t)).
\]
satisfying the initial conditions
\[
\gamma(0) = e
\Textand
\dot{\gamma}(0) = X_0.
\]
Then,
\[
\gamma(t) =  \exp((1-\kappa)t X_0 + \kappa t X_0^T) \,\exp(\kappa t (X_0-X_0^T)).
\]
%\end{shaded}
\end{proposition}
%%%%%%%%%

%%%%%%%
\begin{proof}
Clearly, the initial conditions are satisfied. Differentiating with respect to $t$ we get
\[
\begin{split}
\dot{\gamma}(t) &= (T\gamma)_t(\partial_t) \\
&=(T\gamma)_t([t+s]) \\
&= [\gamma(t+s)] \\
&= [\gamma(t) + \exp((1-\kappa)t X_0 +\kappa t X_0^T) X_0 \exp(\kappa t (X_0-X_0^T)) \, s] \\
&= i_{\ga(t)}\brk{\gamma(t), \exp((1-\kappa)t X_0 +\kappa t X_0^T) X_0 \exp(\kappa t (X_0-X_0^T))} \\
&= i_{\ga(t)}(\gamma(t),\gamma(t) X(t)).
\end{split}
\]
\end{proof}
%%%%%%

%%%%%%%%%%%%%%%%%%%%%%%%%%%%
\section{inverse-invariant metrics on Lie groups} 

%%%%%%%%%%%
\begin{proposition}
\label{prop:inverse_invariant_metrics_Lie}
Let $G$ be a Lie group. A left- (or right-)invariant metric on $G$ is inverse-invariant if and only if it is bi-invariant.
\end{proposition}
%%%%%%%%%%%%

%%%%%%%%%
\begin{proof}
Note that $\inv = R_{s^{-1}}\circ \inv \circ  L_{s^{-1}}$, 
\[ 
R_{s^{-1}}\circ \inv \circ  L_{s^{-1}}(g)=R_{s^{-1}}\circ \inv (s^{-1}g) =R_{s^{-1}}(g^{-1}s) = g^{-1}.   
\]
By the chain rule,
\[ (
d\inv)_s = (dR_{s^{-1}})_e \circ (d\inv)_e \circ (dL_{s^{-1}})_s. 
\]
Since $(d\inv)_e:T_eG \to T_eG$ is the additive inverse operation $(v \mapsto -v)$,
\[  
(d\inv)_s = - (dR_{s^{-1}})_e \circ (dL_{s^{-1}})_s.    
\]
It follows at once that a bi-invariant metric is inverse-invariant.

Conversely, assume the metric is both left- and inverse-invariant. 
Then $\forall s \in G$ , $(dR_s)_e$ is an isometry, hence the metric is right-invariant as well.
\end{proof}
%%%%%%%%%

%Note: The above proposition holds if we assume right-invariance (instead of left-invariance); the same argument shows $(dL_{s^{-1}})_s$ is an isometry for every $s \in G$. Since its inverse is: $(dL_s)_e$ we finished.

%%%%%%%%%
\begin{proposition}
\label{prop:inverse_invariant_distances_Lie}
Let $G$ be a Lie group. A left- (or right-)invariant distance function $d$ on $G$ is inverse-invariant if and only if it is bi-invariant.
%  Let $d$ be a left- and inverse- invariant distance function on a Lie group $G$. Then $d$ is bi-invariant.
\end{proposition}
%%%%%%%%%

%%%%%%%%%
\begin{proof}
Assume $d$ is left- and inverse-invariant. Since $\inv = R_{s^{-1}}\circ \inv \circ  L_{s^{-1}}$, $d$ is also right-invariant.

Conversely, assume $d$ is bi-invariant. Then,
\[ 
d(x,y)=d(1,x^{-1}y)=d(y^{-1},x^{-1})=d(x^{-1},y^{-1}). 
\]
%Since $L_{s^{-1}}, \inv$ are isometries, it follows $R_{s^{-1}}$ is an isometry.
\end{proof}
%%%%%%%%%

In fact, \propref{prop:inverse_invariant_distances_Lie} implies \propref{prop:inverse_invariant_metrics_Lie} by virtue of the Myers-Steenrod theorem whereby every isometry of $d$ is a Riemannian isometry.

%%%%%%%%%%%%%%%%%%%%%%%%%%%%%%%%%%%%%%%%%%%%%%

%\Raz{WE HAVE TO GET RID OF THE NOCITE COMMAND}
%\nocite{Cia05,243767,LNN14,MN14,NEM15,boor1985naive,229549,1623388}

%%%%%%%%%%%%%%%%%%%%%%%%%%%%%%%%%%%%%%%%%%%%%%

{\bfseries Acknowledgements}
We are grateful to Patrizio Neff for clarifying the physical motivation behind the invariance assumptions and for pointing out relevant historic references.
We are very much indebted to Amotz Oppenheim, Pavel Giterman and Amitay Yuval for useful discussions.
We also thank Pietro Majer for useful comments on mathoverflow.

This research was partially supported by the Israel-US Binational Foundation (Grant No. 2010129), by the Israel Science Foundation (Grant No. 661/13), and by a grant from the Ministry of Science, Technology and Space, Israel and the Russian Foundation for Basic Research, the Russian Federation.

\bibliographystyle{unsrt}
\bibliography{./MyBibs}
%\bibliography{/home/asaf/Dropbox/Research/My_Notes/A_short_proof_of_the_geodesic_distance_to_SO/MyBibs}
%%%%%%%%%%%%%%%%%%%%%%%%%%%%%%%%%%%%%%%%%%%%%%

\end{document}